\documentclass[a4paper,11pt]{article}
\usepackage[utf8]{inputenc}
\usepackage[noblocks]{authblk}
\usepackage{amsmath, amssymb, amsthm}
\usepackage{lmodern} 
\usepackage{a4wide} 
\usepackage{xcolor}
\usepackage{graphicx}
\usepackage{hyperref}
\usepackage{enumitem}
\usepackage{mathtools}
\usepackage{esint}
\usepackage[normalem]{ulem}
\usepackage{verbatim}
\usepackage{hyperref}

\newcommand{\cF}{\mathcal F} 
\newcommand{\cE}{\mathcal E}
\newcommand{\cL}{\mathcal L}
\newcommand{\cH}{\mathcal H}

\newcommand{\dd}{\mathrm d}  
\newcommand{\TV}{\mathrm{TV}}
\newcommand{\ROF}{\mathrm{ROF}}

\newcommand{\weaklystar}{\stackrel{*}{\rightharpoonup}}

\DeclareMathOperator{\Div}{div}

\newtheorem{thm}{Theorem}[section]
\newtheorem{lem}[thm]{Lemma} 
\newtheorem{prop}[thm]{Proposition}
\theoremstyle{definition}
\newtheorem{defn}[thm]{Definition}
\theoremstyle{remark}
\newtheorem{remark}[thm]{Remark}

\newcommand{\mres}{
	\,\raisebox{-.127ex}{\reflectbox{\rotatebox[origin=br]{-90}{$\lnot$}}}\,
}

\newcommand{\twopartdef}[4]
{
	\left\{
	\begin{array}{ll}
		#1 & #2 \\
		#3 & #4
	\end{array}
	\right.
}

\mathtoolsset{showonlyrefs}

\author{Wojciech G{\' o}rny\thanks{Faculty of Mathematics, Universit\"at Wien, Oskar-Morgerstern-Platz 1, 1090 Vienna, Austria; Faculty of Mathematics, Informatics and Mechanics, University of Warsaw, Banacha 2, 02-097 Warsaw, Poland; {\tt  wojciech.gorny@univie.ac.at}}\quad 
Micha{\l} {\L}asica\thanks{Institute of Mathematics of the Polish Academy of Sciences, {\'S}niadeckich 8, 00-656 Warsaw, Poland; {\tt mlasica@impan.pl}}\quad
Alexandros Matsoukas\thanks{Department of Mathematics, School of Applied Mathematical and Physical Sciences, National Technical 
University of Athens,  Zografou Campus, 157 80 Athens, Greece; {\tt alexmatsoukas@mail.ntua.gr}}}
\title{
Euler--Lagrange equations for variable-growth total variation
}

\newcommand{\note}[1]
{\vskip.3cm
\fbox{%
\parbox{0.93\linewidth}{\footnotesize #1}}
\vskip.3cm}

\makeatletter
\newcommand{\namelabel}[1]{%
  \phantomsection
  \renewcommand{\@currentlabel}{#1}
  \label{#1}
}
\makeatother

\makeatletter
\newcommand{\labeltext}[2]{%
  \@bsphack
  \csname phantomsection\endcsname 
  \def\@currentlabel{#1}{\label{#2}}%
  \@esphack
}
\makeatother

\def\blue{\color{blue}}
\def\violet{\color{violet}}

\begin{document}

\maketitle 

\begin{abstract}
We consider a class of integral functionals with Musielak--Orlicz type variable growth, possibly linear in some regions of the domain. This includes $p(x)$ power-type integrands with $p(x) \geq 1$ as well as double-phase $p\!-\!q$ integrands with $p=1$. The main goal of this paper is to identify the $L^2$-subdifferential of the functional, including a local characterisation in terms of a variant of the Anzellotti product defined through the Young's inequality. As an application, we obtain the Euler--Lagrange equation for the variant of Rudin--Osher--Fatemi image denoising problem with variable growth regularising term. 
\end{abstract}


\section{Introduction}

In the recent decades, starting with the seminal work of Rudin, Osher and Fatemi \cite{ROF}, image denoising models based on total variation have attracted much attention. The original Rudin--Osher--Fatemi (ROF) denoising model \cite{ROF}, as reformulated by Chambolle and Lions in \cite{ChambolleLions}, amounts to finding the minimiser of
\[ \ROF(v) := \lambda \TV(v) + \frac{1}{2} \int_\Omega |v-f|^2 \,\dd x,\]
where $\Omega$ is a bounded Lipschitz domain, $f \in L^2(\Omega)$ is a given noisy image and $\lambda > 0$ is a suitably chosen constant. Here, $\TV$ denotes the classical total variation given by 
\begin{equation} \label{TV}
 \TV(v) := \sup \left\{ \int v \Div \xi \,\dd x\ \bigg|\ \xi \in C_c^1(\Omega)^n,\ |\xi|\leq 1 \text{ in } \Omega \right\}. 
\end{equation} 
The ROF model enjoys favorable properties such as simplicity and numerical tractability; a numerical implementation can be done e.g.\ using the Chambolle--Pock algorithm \cite{ChambollePock}. While it has a crucial advantage over linear methods in that it is able to preserve sharp contours, on the hand it suffers from \emph{staircasing}, i.e., emergence of piecewise constant structures from the noise.

One possible solution to the staircasing problem is to modify the functional by introducing a variable-growth term in the total variation. Several particular settings were studied in the literature; the variable exponent case was studied in \cite{CLR}, both from the theoretical and numerical perspective, and the double-phase image restoration was introduced in \cite{HarjulehtoHasto2021}. In this paper, we consider a general framework for such models using the notion of variable-growth total variation introduced in \cite{EHH}. Given a bounded Lipschitz domain $\Omega$ and a function $\varphi \in \Phi_c(\Omega) \cap C(\Omega\times [0, \infty))$, see \cite{EHH} or below in Section \ref{sec:preliminaries} for the definition of $\Phi_c(\Omega)$, we define the variable-growth total variation $\TV_\varphi \colon L^1(\Omega) \to [0,\infty]$ by 
\[ \TV_\varphi(v) := \sup \left\{ \int v \Div \xi - \varphi^*(x, |\xi|) \,\dd x \, \bigg| \, \xi \in C_c^1(\Omega)^n \right\}, \]
where $\varphi^*(x,\cdot)$ is the Legendre--Fenchel dual to $\varphi(x, \cdot)$ for $x \in \Omega$. This coincides with the \emph{dual modular} $\varrho_{V, \varphi}$ from \cite{EHH}. In \cite[Theorem 6.4]{EHH}, it is proved that if $\varphi$ satisfies assumptions \ref{A0}, \ref{aDec} and \ref{RVA1} listed below in Section \ref{sec:preliminaries}, then for $v \in BV(\Omega)$
\begin{equation} \label{TV_explicit}
    \TV_\varphi(v) = \varphi(|Dv|)(\Omega)
\end{equation} 
where the Borel measure $\varphi(|Dv|)$ is defined by 
\[ \varphi(|Dv|)(E) = \int_E \varphi(x,|\nabla v|) \,\dd x  + \int_E \varphi_\infty' \,\dd |D^s v|\]
with $\varphi_\infty'(x) = \lim_{t \to \infty} \varphi(x,t)/t$ called the \emph{recession function} of $\varphi$. The functional $\TV_\varphi$ is a generalisation of classical total variation $\TV$. Indeed, setting $\varphi(x,t) = t$ for all $x \in \Omega$ and $t \in [0, \infty)$, one has
\begin{equation} \label{eq:chi} 
\varphi^*(x,t) = \infty \chi_{(1,\infty)}(t): = \twopartdef{0}{\text{if } t \in [0,1];}{\infty}{\text{if } t>1,}
\end{equation} 
and $\varphi(|Dv|) = |Dv|$ for $v \in BV(\Omega)$. We note that the measure $\varphi(|Dv|)$ may take value $\infty$; in fact $\varphi(|Dv|)(E) < \infty$ if and only if $|D^s v|(E \cap \{\varphi_\infty' = \infty\})=0$ and $\int_E \varphi(x,| \nabla u|) \, \dd x < \infty$. Regardless of whether \eqref{TV_explicit} holds, by definition $\TV_{\varphi}$ is convex and lower semicontinuous. Thus, its restriction to $L^2(\Omega)$ is weakly lower semicontinuous. 

In this paper, in order to provide another point of view to combat possible staircasing effects, we investigate the variable-growth total variation $\TV_\varphi$ and associated minimisation problem for 
\[ \ROF_\varphi(v) := \lambda \TV_\varphi(v) + \frac{1}{2} \int_\Omega |v-f|^2 \,\dd x. \]
The functional $\ROF_\varphi$ is weakly lower semicontinuous and coercive on $L^2(\Omega)$, so it has a minimiser $u$. By the strict convexity of $\ROF_\varphi$, it is unique. Denoting by $\cE_\varphi$ the restriction of $\TV_\varphi$ to $L^2(\Omega)$, the minimiser is identified by the abstract Euler--Lagrange equation
\begin{equation} \label{EL_abstract} 
u - f \in - \lambda \partial \cE_\varphi(u).
\end{equation} 
In order to deliver a verifiable optimality condition, one needs to characterise the subdifferential $\partial \cE_\varphi(u)$. This is the main aim of the present paper. Let us note that this also allows us to characterise solutions to the $L^2$-gradient flow of the $\varphi$-total variation; indeed, since the functional is convex and lower semicontinuous, it follows from the general theory of semigroup solutions to gradient flows in Hilbert spaces (introduced in \cite{Brezis0,Komura}; see e.g. \cite{Brezis} for an overview) that there exists a unique strong solution to the $L^2$-gradient flow of the functional $\cE_\varphi$ and the characterisation of its subdifferential again gives a verifiable condition for solutions of this gradient flow.

Our work belongs to a stream of results where the subdifferential of various variants of total variation is characterised (usually with the application to gradient flows in mind). For the classical $\TV$ this can be found in \cite{ABCM2001} (see also \cite{ACMBook}). This was later generalized to various classes of non-autonomous (i.e.\ the integrand depends on $x$) functionals of linear growth in \cite{ACM2002, Moll2005, GornyMazon}. However, the present paper is to our knowledge the first one, where this is carried out in the variable-growth case. The previous works on the non-autonomous case tend to rely on Reshetnyak's continuity theorem, which requires continuity of the recession function $\varphi'_\infty$. Naturally, this fails in our present setting. Instead, we rely on a recent approximation result from \cite{EHH}, which we slightly improve to obtain approximation by functions smooth up to the boundary. Another novelty which facilitates our goal, is to provide a characterisation of the subdifferential in a natural language of convex duality. More precisely, we define a version of the Anzellotti product (cf.\ \cite{Anzellotti}, \cite{ACMBook}) for our setting and show that it satisfies a kind of Young's inequality. Then, elements of the subdifferential are precisely those for which it is an equality.


The structure of this paper is the following: in Section \ref{sec:preliminaries}, we recall the necessary definitions concerning the variable-growth total variation; in Section~\ref{sec:mainresults}, we present our main results on the characterisation of $\partial \cE_\varphi(u)$; and in Section~\ref{sec:particularcases} we discuss several particular cases to which our results may be applied.

\section{Preliminaries}\label{sec:preliminaries}

We begin this section with the definition of the space of functions with bounded variation. We refer to the book \cite{AFP} for further information. The space of functions with bounded variation is defined as 
$$BV(\Omega)\,=\,\{v\in L^{1}(\Omega): \TV (v)<\infty \},$$
where the definition of the total variation $\TV$ is given in \eqref{TV}. Equivalently, it is known that such functions have a distributional derivative $Dv$, which is a bounded Radon measure. In particular, for $v\in BV(\Omega)$ we have 
$$\TV(v)=|Dv|(\Omega),$$
where $|Dv|(\Omega)$ is the total variation of the measure $Dv$. Thus, the expression in \eqref{TV_explicit} is a natural generalization of this fact, for the variable-growth total variation. 
Moreover, by the Lebesgue decomposition of the Radon measure $Dv$, we have $$Dv=\nabla v \, \mathcal{L}^n +D^s v,$$
where $\nabla v$ denotes the Radon--Nikodym derivative of $Dv$ with respect to the Lebesgue measure~$\mathcal{L}^n$, i.e., the absolutely continuous part of $Dv$ while $D^s v$ is the singular part. The space $BV(\Omega)$ is endowed with the norm
$$\|v\|_{BV(\Omega)}=\|v\|_{L^1(\Omega)}+\TV(v).$$

Throughout the paper, we will require that the function denoted by $\varphi$ satisfies (a subset) of the assumptions below. They ensure that the corresponding Musielak--Orlicz spaces are well-defined and have suitable approximation properties. The notation follows the one of \cite{EHH} and \cite{HarjulehtoHasto}. 

\begin{defn}
Let $\varphi: \Omega \times [0,\infty) \rightarrow [0,\infty]$. We say that $\varphi$ is a {\it 
$\Phi$-prefunction}, if the following conditions hold for every $x \in \Omega$:

\begin{enumerate}
\item[(i)] The map $x \mapsto \varphi(x,|f(x)|)$ is measurable for every measurable $f: \Omega \rightarrow \mathbb{R}$; 

\item[(ii)] The map $t \mapsto \varphi(x,t)$ is non-decreasing;

\item[(iii)] The map $t \mapsto \varphi(x,t)$ has the following limits at $0$ and $\infty$:
\begin{equation*}
\varphi(x,0) = \lim_{t \to 0^+} \varphi(x,t) = 0
\end{equation*}
and
\begin{equation*}
\lim_{t \rightarrow \infty} \varphi(x,t) = \infty.
\end{equation*}
\end{enumerate}
\end{defn}

Recall that a function $f: (0,\infty) \rightarrow \mathbb{R}$ is $L$-almost increasing, if there exists $L \geq 1$ such that
\begin{equation*}
f(s) \leq Lf(t) \quad \text{for all} \quad 0 < s \leq t
\end{equation*}
and similarly $f: (0,\infty) \rightarrow \mathbb{R}$ is $L$-almost decreasing, if there exists $L \geq 1$ such that
\begin{equation*}
Lf(s) \geq f(t) \quad \text{for all} \quad 0 < s \leq t.
\end{equation*}
Note that in the case $L = 1$, $1$-almost increasing functions correspond to non-decreasing functions, and $1$-almost decreasing functions correspond to non-increasing functions.

\begin{defn}
Let $\varphi: \Omega \times [0,\infty) \rightarrow [0,\infty]$ be a $\Phi$-prefunction. Then:

\begin{enumerate}
\item[(i)] We say that $\varphi$ is a {\it weak $\Phi$-function}, if the map
$$ t \mapsto \frac{\varphi(x,t)}{t} $$
is $L_1$-almost increasing on $(0,\infty)$ for some $L_1 \geq 1$ independent of $x$. The set of all weak $\Phi$-functions on $\Omega$ is denoted $\Phi_w(\Omega)$.

\item[(ii)] We say that $\varphi$ is a {\it convex $\Phi$-function}, if $\varphi \in \Phi_w(\Omega)$ and the map
$$ t \mapsto \varphi(x,t) $$
is convex and left-continuous for every $x \in \Omega$. The set of all convex $\Phi$-functions on $\Omega$ is denoted $\Phi_c(\Omega)$.
\end{enumerate}
\end{defn}

We now define the Musielak--Orlicz spaces and the corresponding Sobolev spaces, for $\varphi \in \Phi_{c}(\Omega)$. By $L^0 (\Omega)$ we denote the linear space of all measurable functions. Given $\varphi \in \Phi_{c}(\Omega)$, the Musielak--Orlicz space $L^{\varphi}(\Omega)$ is defined by
    $$L^{\varphi}(\Omega)\,=\,\{f\in L^{0}(\Omega)\, | \, \varrho_{\varphi}(\lambda f) < \infty \, \text{ for some } \, \lambda >0 \},$$
    where the modular is given by $\varrho_{\varphi}(f)=\int_{\Omega}\varphi (x,|f(x)|) \, \dd x.$
    When equipped with the Luxemburg norm 
    $$\|f\|_{L^{\varphi}(\Omega)}\,=\,\inf \bigg\{ \lambda > 0: \int_{\Omega}\varphi \left(x,\frac{|f|}{\lambda}\right) \, \dd x \le 1 \bigg\}, $$ it becomes a Banach space. If in addition $\varphi$ satisfies \ref{aDec} (see the next definition), then we can have a direct definition of $L^{\varphi}(\Omega)$ using the modular $\varrho_{\varphi}(\cdot)$:
$$L^{\varphi}(\Omega)\,=\,\{f\in L^{0}(\Omega)\, | \, \varrho_{\varphi}(f) < \infty \}.$$ The corresponding Musielak--Orlicz--Sobolev space $W^{1,\varphi}(\Omega)$ is defined as 
$$W^{1,\varphi}(\Omega)\,=\,\{v \in W^{1,1}(\Omega): |v| , |\nabla v| \in L^{\varphi}(\Omega) \},$$
where $\nabla v$ denotes the weak derivative of $v$. We equip this space with the norm
$$\|v\|_{W^{1,\varphi}(\Omega)}\,=\,\|v\|_{L^{\varphi}(\Omega)}+\|\nabla v\|_{L^{\varphi}(\Omega)}.$$
In the paper \cite{EHH}, the authors introduced an extension of the $BV(\Omega)$ space namely, the space $BV^{\varphi}(\Omega)$ built upon a $\Phi$-function $\varphi$. It is mentioned that the dual modular $\varrho_{V,\varphi}$, which in our terminology is $\TV_{\varphi}$, does not generate the dual norm $V_{\varphi}$ of this space. Let us emphasise that our approach does not explicitly rely on the properties of these function spaces. Our main focus is the $\TV_{\varphi}$ functional, with the two key ingredients being the representation formula~\eqref{TV_explicit} proved in \cite[Theorem 6.4]{EHH}, and the modular approximation by smooth functions \cite[Proposition 7.1]{EHH}.


With these definitions in mind, we are ready to state the required assumptions on $\varphi$. We follow the notation in \cite{EHH}, additionally denoting by \ref{RVA1} the \emph{restricted \ref{VA1} condition} introduced in \cite[Definition 3.1]{EHH}.

\begin{defn}
Let $\varphi: \Omega \times [0,\infty) \rightarrow [0,\infty]$ and $p,q \geq 1$. We define the following conditions:

\begin{enumerate}
\item[(A0)]\labeltext{(A0)}{A0} There exists $\beta \in (0,1]$ such that for every $x \in \Omega$
\begin{equation*}
\varphi(x,\beta) \leq 1 \leq \varphi(x,1/\beta);
\end{equation*}

\item[${\rm (aInc)}_p$]\labeltext{${\rm (aInc)}_p$}{aIncp}\labeltext{${\rm (aInc)}_2$}{aInc2} There exists $L_p \geq 1$ (independent of $x$) such that for every $x \in \Omega$ the map
\begin{equation*}
t \mapsto \frac{\varphi(x,t)}{t^p}
\end{equation*}
is $L_p$-almost increasing in $(0,+\infty)$;

\item[(aInc)]\labeltext{(aInc)}{aInc} The condition \ref{aIncp} holds for some $p > 1$;

\item[${\rm (aDec)}_q$]\labeltext{${\rm (aDec)}_q$}{aDecq}\labeltext{${\rm (aDec)}_2$}{aDec2} There exists $L_q \geq 1$ (independent of $x$) such that for every $x \in \Omega$ the map
\begin{equation*}
t \mapsto \frac{\varphi(x,t)}{t^q}
\end{equation*}
is $L_q$-almost decreasing in $(0,+\infty)$;

\item[(aDec)]\labeltext{(aDec)}{aDec} The condition \ref{aDecq} holds for some $q > 1$;

\item[(A1)]\labeltext{(A1)}{A1} For every $K > 0$, there exists $\beta \in (0,1]$ such that for every $x,y \in \Omega$
\begin{equation*}
\varphi(x,\beta t) \leq \varphi(y,t) + 1 \quad \mbox{if} \quad \varphi(y,t) \in \bigg[0, \frac{K}{|x-y|^n} \bigg];
\end{equation*}

\item[(VA1)]\labeltext{(VA1)}{VA1} For every $K > 0$, there exists a modulus of continuity $\omega$ such that for every $x,y \in \Omega$
\begin{equation*}
\varphi \bigg(x, \frac{t}{1+\omega(x-y)} \bigg) \leq \varphi(y,t) + \omega(|x-y|) \quad \mbox{if} \quad \varphi(y,t) \in \bigg[0, \frac{K}{|x-y|^n} \bigg];
\end{equation*}

\item[(RVA1)]\labeltext{(RVA1)}{RVA1} We have $\varphi \in \Phi_w(\Omega)$, the condition \ref{A1} holds, and for every $K > 0$ there exists a modulus of continuity $\omega$ such that
\begin{equation*}
\varphi \bigg(x, \frac{t}{1+\omega(x-y)} \bigg) \leq \varphi(y,t) + \omega(|x-y|) \quad \mbox{if} \quad \varphi(y,t) \in \bigg[0, \frac{K}{|x-y|^n} \bigg]
\end{equation*}
for every $x,y \in \Omega$ with $\varphi'_\infty(x) < \infty$ or $\varphi'_\infty(y) < \infty$.
\end{enumerate}
\end{defn}

A discussion on these conditions and the role they play for the validity of approximation results in the Musielak--Orlicz spaces may be found in \cite{EHH}; here, let us just note a few small observations. The conditions \ref{aDecq} and \ref{aIncp} compare the rate of growth of $\varphi$ with power functions from above/below, while \ref{aDec} and \ref{aInc} say that the growth is at most/at least polynomial.

The condition \ref{A1} is an almost continuity condition in the first coordinate, which in the variable exponent case corresponds to the log-H\"older continuity of the inverse of the exponent. The 'vanishing \ref{A1}' condition \ref{VA1} and the 'restricted \ref{VA1}' condition \ref{RVA1} are its versions which play a similar role. We later present some examples of these conditions in particular settings.

\begin{remark}
Note that in the definitions above we allow for a $\Phi$-prefunction to attain the value $\infty$. While it is not always assumed in the literature, it is necessary in our setting; in the next sections, while usually $\varphi$ only takes finite values, its conjugate function $\varphi^*$ does not. This is typical in the linear growth setting: for instance, if we consider the isotropic total variation and take $\varphi(x,t) = t$, we have that $\varphi^*(x,t) = \infty \chi_{(1,\infty)}(t)$ (see \eqref{eq:chi}). Note that such $\varphi^*$ fulfills many of the assumptions listed above: \ref{A0} holds for any $\beta \leq 1$; \ref{aIncp} for any $p > 1$ with constant $L_p = 1$; and \ref{VA1} with any modulus of continuity.

While for the reasons given above we do not generally assume finite values of $\varphi$, some of the assumptions listed above are closely related to this fact. Firstly, observe that if $\varphi$ satisfies \ref{A0} and the condition \ref{aDecq} for any $q \geq 1$, which is a natural growth assumption in the present setting, the values of $\varphi$ are necessarily finite. The condition \ref{A0} itself does not require the values to be finite, and whenever $\varphi$ verifies this condition then the conjugate $\varphi^*$ also does, see \cite[Theorem 2.4.8]{HarjulehtoHasto} and \cite[Lemma 3.7.6]{HarjulehtoHasto}. Furthermore, the regularity conditions \ref{A1} and \ref{VA1} {\it assume} finiteness of the value at a given point and state some estimates on the behaviour of $\varphi$ close to this point---therefore, for their statement there is no need to require that the values of $\varphi$ are finite. 
\end{remark}


Let us collect here several simple properties of $\Phi$-functions shown in \cite{HarjulehtoHasto}. First of all, it is a general property that $\varphi^*$ is convex as a result of taking a convex conjugate: more precisely, by \cite[Lemma 2.5.8]{HarjulehtoHasto} it holds that
\begin{equation}\label{eq:phistarisconvex}
\varphi \in \Phi_{w}(\Omega) \Rightarrow \varphi^{*} \in \Phi_{c}(\Omega).
\end{equation}
Moreover, the following implication holds:
\begin{equation}\label{eq:leftcontinuitylowersemicontinuity}
\varphi \in \Phi_{w}(\Omega) \mbox{ is left-continuous }\Rightarrow \varphi \mbox{ is lower semicontinuous}.
\end{equation}
In particular, this holds for $\varphi \in \Phi_{c}(\Omega)$.

Furthermore, supposing that $\mathcal{L}^n(\Omega) < \infty$ and the weak $\Phi$-function $\varphi \in \Phi_w(\Omega)$ satisfies \ref{A0} and \ref{aIncp}, by \cite[Corollary 3.7.9]{HarjulehtoHasto} we have
\begin{equation}\label{eq:inclusionforinc}
L^\varphi(\Omega) \hookrightarrow L^p(\Omega)
\end{equation}
with the estimate
\begin{equation}\label{eq:inclusionforinc2}
c \int_\Omega |f(x)|^p \, \dd x \leq \int_\Omega \varphi(x,|f|(x)) \, \dd x + \mathcal{L}^n(\Omega)
\end{equation}
for some $c = c(\Omega, L_p, \beta) > 0$, where $\beta$ is the constant in condition \ref{A0}.



\section{Results}\label{sec:mainresults}

This section is organised as follows. First, we show several preliminary results concerning the relevant function spaces, in particular an up-to-the-boundary approximation by smooth functions. Then, we provide an abstract characterisation of the subdifferential via convex duality techniques, and subsequently we introduce a notion of Anzellotti-type pairings to provide a verifiable condition for being in the subdifferential.

\subsection{Preliminary lemmata}

Let us denote by $\cE_\varphi$ the restriction of $\TV_\varphi$ to $L^2(\Omega)$. Let $D_\varphi$ denote the domain of $\cE_\varphi$, i.\,e., 
\begin{equation}
D_\varphi := \left\{ v \in L^2(\Omega): \, \cE_\varphi(v) < \infty \right\}.
\end{equation}
Then, under natural assumptions on the growth of $\varphi$, we have a Poincar\'e-type inequality which implies the following inclusion.

\begin{prop}\label{prop:poincare} 
Suppose that $\varphi \in \Phi_c \cap C(\Omega \times [0, \infty))$ satisfies the conditions \ref{A0}, \ref{RVA1} and \ref{aDec}, and that $\varphi^*$ satisfies the condition \ref{VA1}. Then,
$$D_\varphi \subset L^\varphi(\Omega).$$ 
\end{prop}

\begin{proof} 
Given $v \in D_\varphi$ and $m>0$, let $T_m v \in L^\infty(\Omega) \subset L^\varphi(\Omega)$ denote the simple truncation of~$v$ at level $m$, i.e.,
$T_m: \mathbb{R} \rightarrow \mathbb{R}$ is defined as
$$T_m(r):= \left\{ \begin{array}{lll} m \quad &\hbox{if} \ \ r > m; \\[5pt] r \quad &\hbox{if} \ \ \vert r \vert \leq m; \\[5pt] - m \quad &\hbox{if} \ \ r < -m. \end{array} \right.$$
Fix any $m > 0$. By \cite[Theorem 3.99]{AFP}, it holds that $T_m v \in BV(\Omega)$ and
\begin{equation}\label{D_trunc}
D T_m v = \chi_{\{-m<v<m\}}\nabla v \mathcal{L}^{n} + (T_m(v^{+}) - T_m(v^{-})) \nu_v \, \cH^{n-1} \mres J_v +  \chi_{\{-m<\tilde{v}<m\}} D^{c} v, 
\end{equation} 
where $\tilde{v}$ is the approximate limit of $v$. To see that $\cE_\varphi(T_m v) \leq \cE_\varphi(v)$, recall that the total variation defines a Borel measure $\varphi(|Dv|)$ in $\Omega$. For all $m>0$, denote 
$$\Omega_{m} = \Omega \cap \{ -m < v < m \}.$$
The approximate jump points of $T_{m}(v)$ appear in the sets $\Omega_{m}$ because the truncated functions $T_{m}(v)$ are constant elsewhere. More precisely, we have
\begin{equation} 
(T_m(v^{+}) - T_m(v^{-})) \nu_v \, \cH^{n-1} \mres J_v = (v^{+} - v^{-}) \nu_v \, \cH^{n-1} \mres J_v\cap \{-m<\tilde{v}<m\}
\end{equation}
which yields that 
$$\cE_\varphi(T_m v)=\varphi(|Dv|)(\Omega_{m}) \leq \varphi(|Dv|)(\Omega)=\cE_\varphi(v),$$
so the desired claim holds.

By \cite[Proposition 7.1]{EHH}, for every $m \in \mathbb{N}$ there exists a sequence $(v^k_m) \subset C^\infty(\Omega)$ such that 
\begin{equation} \label{trunc_conv} 
v^k_m \to T_m v \text{ in } L^\varphi(\Omega) \quad \text{and} \quad \cE_\varphi(v^k_m) = \int_\Omega \varphi(x,| \nabla v^k_m|)\, \dd x \to \cE_\varphi(T_m v) 
\end{equation}
as $k \to \infty$. Without loss of generality all $v^k_m$ belong to $D_\varphi$. From the Poincar{\' e}--Sobolev inequality in the form given in \cite[Theorem 6.2.8]{HarjulehtoHasto}, we deduce
\[ \|v^k_m\|_{L^\varphi(\Omega)} \leq C \left( \|\nabla v^k_m\|_{L^\varphi(\Omega)} + \|v^k_m\|_{L^1(\Omega)}\right). \]
By \ref{aDec}, we have 
\[\|\nabla v^k_m\|_{L^\varphi(\Omega)} \leq \max\left\{\cE_\varphi(v^k_m), (L_q \cE_\varphi(v^k_m))^{1/q}\right\} \]
with $q\geq 1$. Thus, passing to the limit $k \rightarrow \infty$ using \eqref{trunc_conv},  
\begin{multline} \label{trunc_bound} \|T_m v\|_{L^\varphi(\Omega)} \leq C \left( \max\left\{\cE_\varphi(T_m v), (L_q \cE_\varphi(T_m v))^{1/q}\right\} + \|T_m v\|_{L^1(\Omega)}\right)\\ \leq C \left( \max\left\{\cE_\varphi(v), (L_q \cE_\varphi(v))^{1/q}\right\} + \|v\|_{L^1(\Omega)}\right).
\end{multline} 
Since $\varphi \in \Phi_{c}(\Omega)$, it has the Fatou property by \cite[Lemma 3.3.8 b)]{HarjulehtoHasto}. Letting $m \rightarrow \infty$, we obtain that $v \in L^\varphi(\Omega)$.
\end{proof} 

\begin{remark} By the explicit expression of $\cE_\varphi$ given in \cite[Theorem 6.4]{EHH} and the above proposition, we see that the effective domain $D_\varphi$ of $\cE_\varphi$ consists of those functions $v\in BV(\Omega)\cap L^2(\Omega)$ such that $v \in L^\varphi (\Omega), |\nabla v| \in L^\varphi (\Omega)$ and $|D^s v| (\{\varphi_\infty' =\infty \})=0$.   
\end{remark}




In the sequel, we assume that $\varphi$ satisfies the conditions \ref{A0}, \ref{RVA1} and \ref{aDec}, and that its conjugate $\varphi^*$ satisfies \ref{VA1}. In particular, $\varphi$ satisfies \ref{aDecq} with some (finite) $q \geq 2$. Let us note that these are exactly the conditions needed to apply the approximation by smooth functions result from \cite{EHH}. Consider the set 
\[ X_{\varphi^*,2} = \left\{\xi \in L^1(\Omega)^n\ \bigg|\ \int_\Omega \varphi^*(x,|\xi|)\, \dd x < \infty, \ \Div \xi \in L^2(\Omega)\right\}.\]
Once $\varphi$ satisfies \ref{A0} and \ref{aDecq} with $q \geq 2$, its conjugate function $\varphi^*$ satisfies \ref{A0} and \ref{aIncp} with $1 < p \leq 2$, where $\frac1p + \frac1q = 1$, so by the estimate \eqref{eq:inclusionforinc2} applied for $\varphi^*$ the condition $\int_\Omega \varphi^*(x,|\xi|)\, \dd x < \infty$ implies $\xi \in L^p(\Omega)^n$ with some $1<p\leq 2$. In particular, 
\[ X_{\varphi^*,2} \subset X_{p,p} := \bigg\{ \xi \in L^p(\Omega)^n\ \bigg|\ \Div \xi \in L^p(\Omega) \bigg\}.\]
The significance of this inclusion is the following. Let us recall that there exists a bounded linear operator $X_{p,p} \ni \xi \mapsto \nu\cdot \xi \in W^{1-1/q,q}(\partial \Omega)^*$ that coincides with the usual normal boundary trace for $\xi$ smooth, see \cite[Lemma 3.10]{NovotnyStravskraba}. The space $W^{1-1/q,q}(\partial \Omega)$ is the trace space of $W^{1,q}(\Omega)$, and the Gauss--Green formula 
\begin{equation}\label{Green}
     \int_\Omega v \Div \xi + \int_\Omega \xi \cdot \nabla v = \langle \nu \cdot \xi, v|_{\partial\Omega} \rangle_{W^{1-1/q,q}(\partial \Omega)^*,W^{1-1/q,q}(\partial \Omega)} 
\end{equation}
holds whenever $\xi \in X_{p,p}$ and $v \in W^{1,q}(\Omega)$.  In particular, every vector field in $X_{\varphi^*,2}$ admits a normal trace, and the Gauss--Green formula holds whenever $\xi \in X_{\varphi^*,2}$ and $v \in W^{1,q}(\Omega)$; we will heavily rely on this fact in the characterisation of the subdifferential of $\cE_\varphi$.

\begin{remark} 

In the definition of the class of admissible vector fields $X_{\varphi^*,2}$, we could have assumed that $\xi \in L^{\varphi^{*}}(\Omega)^n$ instead of finiteness of the modular $\int_{\Omega} \varphi^{*}(x,|\xi|) \, \dd x $.  Since we do not assume that $\varphi^{*}$ satisfies \ref{aDec}, $L^{\varphi^{*}}(\Omega)^n$ may contain elements with infinite modular, so this would be more inclusive, thus strengthening a little some of our auxilliary results. On the other hand, this matters little in the context of the main goal of this paper, since vector fields that appear in the characterization of the subdifferential naturally have finite modular.
    
\end{remark}




In the proof of Proposition \ref{prop:poincare}, we used \cite[Proposition 7.1]{EHH}, where a suitable approximating sequence in $C^\infty(\Omega)$ was obtained for any $v \in D_\varphi$. For our purposes, we will need approximating sequences to be smooth up to the boundary. 


\begin{lem}\label{lem:approx_improv}
Suppose that $\varphi \in \Phi_c \cap C(\Omega \times [0, \infty))$ satisfies \ref{A0}, \ref{RVA1} and \ref{aDec}, and $\varphi^*$ satisfies \ref{VA1}. For any
 $v \in D_\varphi$, there exists a sequence $v^k \in C^\infty(\overline{\Omega})$ such that 
 \[v^k \to v  \text{ in } L^2(\Omega) \text{ and in } L^\varphi(\Omega), \qquad \int_\Omega \varphi(x,|\nabla v^k|) \,\dd x \to \cE_\varphi(v).\] 
\end{lem}

\begin{proof}





By \cite[Proposition 7.1]{EHH}, there exists a sequence $\widetilde{v}^k \in C^\infty(\Omega)$ such that the postulated convergences hold. Thus, using a diagonal argument, we see that it is enough to prove the assertion for $v \in C^\infty(\Omega)\cap D_\varphi$. In particular, we can assume that $v \in W^{1, \varphi}(\Omega) \cap L^2(\Omega)$. Since $\Omega$ is bounded, assumption (A2) appearing in the paper \cite{Juusti} is satisfied by $\varphi$ \cite[Lemma 4.2.3]{HarjulehtoHasto}. Thus, by \cite[Corollary 4.6]{Juusti}, there exists a sequence $v^k \in C^\infty(\overline{\Omega})$ such that $v^k \to v$ in $W^{1,\varphi}(\Omega)$. Analysing the reasoning in \cite{Juusti}, we also have $v^k \to v$ in $L^2(\Omega)$. By convexity of $\varphi(x, |\cdot|)$, we have
\[ \varphi(x, |\nabla v^k|) \leq (1-t) \varphi(x, |\nabla v|/(1-t)) + t \varphi(x, |\nabla v^k - \nabla v|/t)  \] 
for a.e.\ $x \in \Omega$ and all $0 < t < 1$. By convergence of $v^k$ in $W^{1,\varphi}(\Omega)$, there exists a sequence $a^k >0$, $a^k \to 0$ as $k \to \infty$, such that \[\int_\Omega \varphi(x, |\nabla v^k - \nabla v|/a^k) \, \dd x \leq 1 \quad \text{for } k=1,2,\ldots\]
Thus, since $t \mapsto \varphi(\cdot, t)/t$ is $L_1$-almost increasing, we have for a fixed $t$ and $k$ large enough 
\[ t\int_\Omega \varphi(x, |\nabla v^k - \nabla v|/t) \, \dd x \leq L_1 a^k \int_\Omega \varphi(x, |\nabla v^k - \nabla v|/a^k) \, \dd x \leq L_1 a^k \to 0 \quad \text{as } k\to \infty.\]
Hence, we obtain
\begin{equation} \label{approx_limsup_est}\limsup_{k \to \infty} \int_\Omega \varphi(x, |\nabla v^k|) \, \dd x \leq  (1-t) \int_\Omega\varphi(x, |\nabla v|/(1-t)) \, \dd x
\end{equation} 
for any $0 < t <1$. Since $\varphi$ is $L_q$-almost decreasing, 
\[\varphi(x, |\nabla v|/(1-t)) \leq \frac{L_q}{(1-t)^q} \varphi(x, |\nabla v|) \leq 2^q L_q \varphi(x, |\nabla v|)\]
for $0< t \leq 1/2$. Since the right-hand side is integrable by assumption, we can pass to the limit $t \to 0^+$ in \eqref{approx_limsup_est} using dominated convergence, obtaining \begin{equation*} 
\limsup_{k \to \infty} \int_\Omega \varphi(x, |\nabla v^k|) \, \dd x \leq  \int_\Omega\varphi(x, |\nabla v|) \, \dd x = \cE_\varphi(v). 
\end{equation*}   
As on the other hand $\cE_\varphi$ is lower semicontinuous, this concludes the proof.
\end{proof}

\subsection{Abstract characterisation}

Our next aim is to provide an abstract characterisation of the subdifferential of the functional $\cE_\varphi$, and through it, a description of the subdifferential of the functional $\ROF_\varphi$ (a local characterisation is provided in the next section). To this end, we first recall the definition of a convex conjugate. It is enough for our purposes to restrict to the Hilbertian setting; indeed, we will only apply it in the case $X= L^2(\Omega)$. For more details we refer to \cite{EkelandTemam}. 

\begin{defn}
Given a Hilbert space $X$ with inner product $(\cdot, \cdot)_X$ and a convex function $F\colon X \rightarrow \mathbb{R} \cup \{ \infty \}$, we define its {\it Legendre--Fenchel transform}\index{Legendre-Fenchel transform} $F^*\colon X^* \rightarrow \mathbb{R} \cup \{ \infty \}$ by the formula
\begin{equation*}
F^*(v) = \sup_{u \in X}\  ( v,u )_X  - F(u) .
\end{equation*}
\end{defn}

In order to characterise the subdifferential of~$\cE_\varphi$, we will use the following well-known result called the Fenchel extremality condition.


\begin{lem} \label{lem:fenchel} 
Suppose that $\cE\colon X \to [0, \infty]$ is convex and lower semicontinuous. Then, 
	\[ w \in \partial \cE(v) \ \Leftrightarrow \ (v,w)_X = \cE(v) + \cE^*(w). \] 
\end{lem} 

To apply the above result, it remains to identify $\cE_\varphi^*$. For this purpose, we introduce the following functionals on $L^2(\Omega)$:
\[\widetilde{\cF}(w) := \inf \left\{ \int_\Omega \varphi^*(x, |\xi|) \,\dd x \ \bigg| \ \xi \in C_c^1(\Omega)^n, \ \Div \xi = w \right\} \]
and
\[\cF(w) := \inf \left\{ \int_\Omega \varphi^*(x, |\xi|) \,\dd x\ \bigg|\ \xi \in L^1(\Omega)^n, \ \Div \xi = w, \ \nu\cdot\xi = 0 \right\}.\]
Clearly, $\cF \leq \widetilde{\cF}$, since the class of admissible vector fields is larger in the definition of $\cF$. Furthermore, we observe that $\cE_\varphi = \widetilde{\cF}^*$: 
\begin{align*} 
\widetilde{\cF}^*(v) &= \sup_{w \in L^2(\Omega)} \bigg[ \int_\Omega v\, w \, \dd x - \inf_{\substack{\xi \in C^1_c(\Omega)^n \\ \Div \xi = w}} \bigg[ \int_\Omega \varphi^*(x, |\xi|) \, \dd x \bigg] \bigg] \\
&= \sup_{w \in L^2(\Omega)} \sup_{\substack{\xi \in C^1_c(\Omega)^n \\ \Div \xi = w}} \bigg[ \int_\Omega v \, w \, \dd x - \int_\Omega \varphi^*(x, |\xi|) \, \dd x \bigg] \\ 
&= \sup_{\xi \in C^1_c(\Omega)^n} \bigg[ \int_\Omega v \Div \xi \, \dd x - \int_\Omega \varphi^*(x, |\xi|) \, \dd x \bigg] = \cE_\varphi(v).
\end{align*}
Our next goal is to show that $\cE_\varphi^* = \mathcal{F}$; for this, we first need the following result.



\begin{lem}\label{lem:F_prop} 
Suppose that $\varphi \in \Phi_c$ satisfies \ref{A0} and \ref{aDec}.
The functional $\cF$ is convex and lower semicontinuous on $L^2(\Omega)$. Moreover, whenever $\cF(w) < \infty$, the infimum in the definition of $\cF(w)$ is attained.  
\end{lem} 


\begin{proof} 
{\bf Step 1.}  First, we show that if $\mathcal{F}(w) < \infty$, it follows from the direct method of the calculus of variations that the infimum is attained. By definition, we have that
\begin{equation*}
m:=\inf \bigg\{ \int_{\Omega} \varphi^{*}(x,|\xi |) \, \dd x \ \bigg| \ \xi \in C_{w} \bigg\} < \infty , 
\end{equation*}
where 
\begin{equation*}
C_{w}=\bigg\{ \displaystyle \xi \in X_{\varphi^*,2}\ \bigg| \ \mathrm{div} \, \xi = w, \ \nu \cdot \xi =0 \bigg\}.
\end{equation*}
The set $C_{w}$ is convex: $X_{\varphi^*, 2}$ is a convex set and both constraints in the definition are linear. Moreover, since the functional $\mathcal{F}$ is by definition bounded from below (by zero), there exists a sequence $\xi^{k}\in C_{w}$ such that 
\begin{equation*}
\int_{\Omega} \varphi^{*}(x,|\xi^{k}| ) \, \dd x \to m.
\end{equation*}
In particular, the sequence $\left\{\int_{\Omega} \varphi^{*}(x,|\xi^{k}| ) \, \dd x \right\}$ is bounded, i.e., there exists $C>0$ such that 
$$\int_{\Omega} \varphi^{*}(x,|\xi^{k}| ) \, \dd x \le C $$ 
for all $k \in \mathbb{N}$. Since $\varphi$ satisfies \ref{A0} and \ref{aDec}, we have that $\varphi^*$ satisfies \ref{A0} and \ref{aInc}, and consequently by the estimate \eqref{eq:inclusionforinc2} applied for $\varphi^*$ we get
\begin{equation*}
\|\xi^{k} \|_{L^p(\Omega)^n} \le C.
\end{equation*}
Therefore, by the Banach--Alaoglu theorem there exists $\xi^{*} \in L^{p}(\Omega )^{n}$ such that, up to taking a (not relabeled) subsequence, we have
\begin{equation*}
\xi^{k} \rightharpoonup \xi^{*} \text{ in } L^{p}(\Omega )^{n}.
\end{equation*}
By property \eqref{eq:phistarisconvex}, the function $\xi \mapsto \varphi^{*}(x,|\xi|)$ is left-continuous for a.e.\ $x\in \Omega$, and hence \cite[Lemma 2.1.5]{HarjulehtoHasto} implies that it is lower semicontinuous.

It is also convex and nonnegative, thus \cite[Theorem 3.20]{Dacorogna} yields
\begin{equation*}
\liminf_{k \rightarrow \infty} \int_{\Omega} \varphi^{*}(x,|\xi^{k}|) \, \dd x \ge \int_{\Omega} \varphi^{*}(x,|\xi^{*}|) \, \dd x. 
\end{equation*}
Hence,
\begin{align*}
\inf \left\{ \int_{\Omega} \varphi^{*}(x,|\xi |) \, \dd x : \, \xi \in C_{w} \right\} &= \liminf_{k \rightarrow \infty} \int_{\Omega}\varphi^{*}(x,|\xi^{k}|) \, \dd x \\ 
&\ge \int_{\Omega}\varphi^{*}(x,|\xi^{*}|) \, \dd x 
\ge \inf \left\{ \int_{\Omega} \varphi^{*}(x,|\xi |) \, \dd x : \, \xi \in C_{w} \right\},
\end{align*}
so all the inequalities above are in fact equalities, which means that
$$ \mathcal{F}(w) = \int_{\Omega} \varphi^{*}(x,|\xi^{*} |) \, \dd x. $$


{\bf \flushleft Step 2.} { The convexity of the functional $\mathcal{F}$ is clear; it remains to show that $\mathcal{F}$ is lower semicontinuous in $L^{2}(\Omega)$.} To this end, let $w^k \to w$ in $L^{2}(\Omega)$. We may assume that
$$ { m := \liminf_{k \rightarrow \infty} \mathcal{F}(w^{k}) < +\infty},$$
otherwise there is nothing to prove. Up to taking a subsequence (still denoted by $w^k$), we may require that
$$ \lim_{k \rightarrow \infty} \mathcal{F}(w^{k}) = m$$
and that
$$ \mathcal{F}(w^{k}) \leq C$$
for all $k \in \mathbb{N}$ and some $C > 0$. In particular, since the infimum in the definition of $\mathcal{F}$ is attained, there exist $\xi^{k}\in C_{w^{k}}$ such that
$$ \mathcal{F}(w^k) = \int_{\Omega} \varphi^{*}(x,|\xi^{k}|) \, \dd x.$$
Consequently, for all $k \in \mathbb{N}$
$$ \int_{\Omega} \varphi^{*}(x,|\xi^{k}|) \, \dd x \leq C.$$
Arguing as in Step 1, we see that there exists $\xi^{*} \in L^{p}(\Omega)^{n}$ such that, possibly passing to a subsequence (still denoted by $\xi^k$),  
$$ \xi^{k}\rightharpoonup \xi^{*} \text{ in } L^{p}(\Omega )^{n} $$
and
\begin{equation*}
\liminf_{k \rightarrow \infty} \int_{\Omega} \varphi^{*}(x,|\xi^{k}|) \, \dd x \geq \int_{\Omega} \varphi^{*}(x,|\xi^{*}|) \, \dd x. 
\end{equation*}
Since $\xi^{k} \in C_{w^{k}}$, we have that $\mathrm{div} \, \xi^{k} = w^{k}$; from the convergence $w^{k} \rightarrow w$ in $L^2(\Omega)$ and $\xi^{k} \rightharpoonup \xi^{*}$ in $L^{p}(\Omega )^{n}$ we deduce that $\mathrm{div} \, \xi^{*} = w$. By the continuity of the operator $\xi \mapsto \nu \cdot \xi$, we get that $\nu \cdot \xi^{*} =0$ and hence $\xi^{*} \in C_{w}$. Since the values of $\mathcal{F}(w^{k})$ are finite, the infimum in the definition of $\mathcal{F}$ is attained, therefore
$$
m = \lim_{k \rightarrow \infty} \mathcal{F}(w^{k}) = \liminf_{k \rightarrow \infty} \mathcal{F}(w^{k}) = \liminf_{k \rightarrow \infty} \int_{\Omega}\varphi^{*}(x,|\xi^{k}|) \, \dd x \geq \int_{\Omega}\varphi^{*}(x,|\xi^{*}|) \, \dd x = \mathcal{F}(w),$$ 
so at the level of the original sequence we have that
$$ m = \liminf_{k \rightarrow \infty} \mathcal{F}(w^{k}) \geq \mathcal{F}(w), $$
which means that $\mathcal{F}$ is lower semicontinuous.
\end{proof}

\begin{lem} \label{lem:dual} 
Suppose that $\varphi \in \Phi_c\cap C(\Omega \times [0, \infty))$ satisfies \ref{A0}, \ref{RVA1} and \ref{aDec}, and $\varphi^*$ satisfies \ref{VA1}. Then $\cF = \cE_\varphi^*$. 
\end{lem} 

\begin{proof} 
Since $\widetilde{\cF} \geq \cF$, we have $\cE_\varphi = \widetilde{\cF}^* \leq \cF^*$. On the other hand, for $v \in D_\varphi \cap C^\infty({\overline{\Omega}})$ and $\xi \in X_{\varphi^*,2}$ with $\nu \cdot \xi = 0$,  
\[ \int_\Omega v \Div \xi \, \dd x = - \int_\Omega\xi \cdot \nabla v \, \dd x \leq \int_\Omega \varphi(|\nabla v|) \, \dd x +  \int_\Omega \varphi^*(x, |\xi|) \, \dd x. \]
Thus, by Lemma \ref{lem:approx_improv}, we deduce for any $v \in D_\varphi$ 
\begin{equation}\label{eq:inequalityforvxi}
\int_\Omega v \Div \xi \, \dd x  \leq \varphi(|D v|)(\Omega) +  \int_\Omega \varphi^*(x,|\xi|) \, \dd x.
\end{equation}
Taking infimum over all $\xi \in  X_{\varphi^*,2}$ with $\nu \cdot \xi = 0$ such that $\Div \xi = w$ for a given $w \in L^2(\Omega)$, we obtain
\[(v,w)_{L^2(\Omega)} \leq \cE_\varphi(v) + \cF(w),\]
whence $\cE_\varphi \geq \cF^*$.  
\end{proof} 



As a consequence, we get the following characterisation of the subdifferential of $\cE_\varphi$.

\begin{thm}\label{thm:abstractcharacterisation}
Assume that $\varphi \in \Phi_c\cap C(\Omega \times [0, \infty))$ satisfies \ref{A0}, \ref{aDec} and \ref{RVA1}, and that $\varphi^*$ satisfies \ref{VA1}. Let $v \in L^2(\Omega)$.  
Then, the following conditions are equivalent:
\begin{itemize}
\item[(a)] $w \in \partial \cE_\varphi(v)$;

\item[(b)] $v \in BV(\Omega)$ and there exists $\xi \in X_{\varphi^*,2} \subset L^1(\Omega)^n$ with $\mathrm{div} \, \xi = w$ and $\xi \cdot \nu = 0$ such that
\begin{equation}\label{eq:equalityincharacterisation}
\int_\Omega v w \, \dd x = \int_\Omega \varphi(|Dv|) + \int_\Omega \varphi^*(x,|\xi|) \, \dd x.
\end{equation}
\end{itemize}
\end{thm}

\begin{proof}
We first show the implication $(a) \Rightarrow (b)$. By Lemma \ref{lem:fenchel}, we have
\begin{equation}\label{eq:equalityforvw1}
\int_\Omega v w \, \dd x = \cE_\varphi(v) + \mathcal{F}(w),
\end{equation}
which under the current assumptions reduces to
\begin{equation*}
\int_\Omega v w \, \dd x = \int_\Omega \varphi(|Dv|) + \inf \left\{ \int_\Omega \varphi^*(x, |\xi|) \,\dd x\ \bigg|\ \xi \in L^1(\Omega)^n, \ \Div \xi = w, \ \nu\cdot\xi = 0 \right\}.
\end{equation*}
Note that if $w \in \partial \cE_\varphi(v)$, by definition $w$ lies in $L^2(\Omega)$ and therefore the left-hand side in \eqref{eq:equalityforvw1} is finite; therefore, the right-hand side is also finite. In particular, $v \in D_\varphi$. Moreover, whenever $\mathcal{F}(w)$ is finite, by Lemma \ref{lem:F_prop} there exists $\xi \in L^1(\Omega)^n$ for which the infimum is attained. We fix such a $\xi$. In particular, $\mathrm{div} \, \xi = w$, $\xi \cdot \nu = 0$ and therefore
\begin{equation}
\int_\Omega v w \, \dd x = \int_\Omega \varphi(|Dv|) + \int_\Omega \varphi^*(x, |\xi|) \,\dd x.
\end{equation}
This concludes the proof of the first implication.

To see the implication $(b) \Rightarrow (a)$, we first observe that $v \in D_\varphi$ and $\mathcal{F}(w) < \infty$; otherwise the equality \eqref{eq:equalityincharacterisation} would be violated. Now, take any $\xi' \in L^1(\Omega)^n$ which is admissible in the definition of $\mathcal{F}$. Then, arguing as in Lemma \ref{lem:dual}, we arrive at the inequality \eqref{eq:inequalityforvxi}, i.e.,
\begin{equation}
\int_\Omega v w \, \dd x = \int_\Omega v \Div \xi \, \dd x  \leq \varphi(|D v|)(\Omega) +  \int_\Omega \varphi^*(x,|\xi'|) \, \dd x.
\end{equation}
Thus, comparing this to equality \eqref{eq:equalityincharacterisation}, we see that
\begin{equation}
\int_\Omega \varphi^*(x,|\xi|) \, \dd x \leq \int_\Omega \varphi^*(x,|\xi'|) \, \dd x
\end{equation}
and thus the minimum in the definition of $\mathcal{F}$ is achieved at $\xi$. Hence,
\begin{equation*}
\int_\Omega v w \, \dd x = \cE_\varphi(v) + \mathcal{F}(w),
\end{equation*}
and Lemma \ref{lem:fenchel} implies that $w \in \partial \cE_\varphi(v)$.
\end{proof}

\subsection{Local characterisation}

In this section, we move towards the main goal of this paper and provide a local characterisation of the subdifferential of the functional $\cE_\varphi$. More precisely, we describe the subdifferential $\partial \cE_\varphi(v)$ in terms of existence of a calibrating vector field $\xi$ satisfying suitable compatibility conditions with the function $v$. To this end, we construct Anzellotti-type pairings (see~\cite{Anzellotti}), which will allow us to give an almost pointwise characterisation of the subdifferential in Theorem \ref{thm:localcharacterisation}.

 
\begin{thm} 
Suppose that $\varphi \in \Phi_c\cap C(\Omega \times [0, \infty))$ satisfies \ref{A0}, \ref{RVA1} and \ref{aDec}, and $\varphi^*$ satisfies \ref{VA1}. Let $v \in D_\varphi$ and $\xi \in X_{\varphi^*, 2}$. Then, there exists a Radon measure $(\xi, Dv)$ on $\Omega$ uniquely defined by \begin{equation} \label{anzellotti_def1} 
\langle(\xi, Dv), \psi\rangle_{C_0(\Omega)^*} := - \int_\Omega \psi\, v\, \Div \xi \, \dd x  - \int_\Omega v\, \xi \cdot \nabla \psi\, \dd x   
\end{equation}   
for any test function $\psi \in C^1_c(\Omega)$. Moreover, 
\begin{equation} \label{eq:anzellotti_young} 
|(\xi, Dv)| \leq \varphi( |Dv|) + \varphi^*(\cdot,|\xi| ) \, \cL^n \quad \text{as measures,}
\end{equation} 
i.e., this inequality holds for any Borel set.
\end{thm} 

\begin{proof}
By Proposition \ref{prop:poincare}, it holds that $v\in D_\varphi \subset L^\varphi(\Omega)$. Since $\xi \in X_{\varphi^*, 2}$, the integrals on the right-hand side of \eqref{anzellotti_def1} are well defined. If $v \in D_\varphi \cap C^1(\Omega)$, we can integrate back by parts  
\[\langle(\xi, Dv), \psi\rangle = - \int_\Omega \psi\, v\, \Div \xi \, \dd x - \int_\Omega v\, \xi \cdot \nabla \psi\, \dd x = \int_\Omega \psi\, \xi \cdot \nabla v \,\dd x\]
 and estimate by Young's inequality
\begin{equation} \label{eq:young_est_loc} \left|\langle(\xi, Dv), \psi\rangle\right| \leq \int_U |\psi|\, |\xi| \, |\nabla v|\, \dd x \leq \int_\Omega |\psi|\, \varphi(x,|\nabla v|) \, \dd x + \int_\Omega |\psi|\, \varphi^*(x,|\xi|) \, \dd x. 
\end{equation} 
By Lemma \ref{lem:approx_improv}, for arbitrary $v \in D_\varphi$, we can take a sequence $(v^k) \subset C^\infty(\Omega)$ that approximates $v$ in modular and in $L^2(\Omega)$. Then, for the first term in the definition of $\langle(\xi, D v^k), \psi\rangle$, since $\psi \in C_c^1(\Omega)$, $\mathrm{div}(\xi) \in L^2(\Omega)$ and $v_k \rightarrow v$ in $L^2(\Omega)$, we obtain
\begin{equation*}
\lim_{k\to +\infty}  \int_\Omega \psi\, v^k \, \Div \xi \, \dd x  =   \int_\Omega \psi\, v \, \Div \xi \, \dd x.
\end{equation*}
As for the second term, since $\psi \in C_c^1(\Omega)$, $\xi \in X_{\varphi^*, 2}$ and $v_k \rightarrow v$ in $L^\varphi(\Omega)$, we obtain
\begin{equation*}
\lim_{k\to +\infty} \int_\Omega v^k \, \xi \cdot \nabla \psi\, \dd x =  \int_\Omega v\, \xi \cdot \nabla \psi\, \dd x.
\end{equation*}
Consequently, $\langle(\xi, Dv^k), \psi\rangle \to \langle(\xi, Dv), \psi\rangle.$ The convergence of $v^k$ to $v$ in modular means that
$$\int_\Omega \varphi(x,|\nabla v^k|) \, \dd x \to \int_\Omega \varphi(|Dv|).$$
On the other hand, by \cite[Theorem 6.4]{EHH}, 
\[\liminf_{k \to \infty} \int_U \varphi(x,|\nabla v^k|) \, \dd x \geq \int_U \varphi(|D v|)\] 
for any open $U\subset \Omega$. Thus, by \cite[Proposition 1.80]{AFP}, $\varphi(|D v^k|) \weaklystar \varphi(|Dv|)$. We deduce from \eqref{eq:young_est_loc} that 
\begin{equation} \label{eq:young_est_loc_lim} \left|\langle(\xi, Dv), \psi\rangle\right|  \leq \int_\Omega |\psi|\,\dd \varphi(|D v|) + \int_\Omega |\psi|\varphi^*(x,|\xi|) \, \dd x. 
\end{equation}
In particular, 
\begin{equation}\label{eq:estimateforpairingwithpsi}
\left|\langle(\xi, Dv), \psi\rangle\right| \leq  \|\psi\|_{C_0(\Omega)} \left(\cE_\varphi(v) + \int_\Omega \varphi^*(x,|\xi|) \, \dd x\right)
\end{equation}
for all $\psi \in C_c^1(\Omega)$. Thus, \eqref{anzellotti_def1} uniquely defines (by density) an element $(\xi, Dw)$ of $C_0(\Omega)^*$, i.e., a finite Radon measure. Now, restricting ourselves to $\psi$ such that $\mathrm{supp}\, \psi \subset U$ and $|\psi| \leq 1$, we get from \eqref{eq:young_est_loc_lim} that 
\begin{equation}\label{eq:estimateforopensubs}
\left| \int_U (\xi, Dv) \right| \leq \int_U |(\xi, Dv)|   \leq  \int_U \varphi(|D v|) + \int_U \varphi^*(x,|\xi|) \, \dd x,
\end{equation}
for all open sets $U\subset \Omega$. From here, appealing to outer regularity property of Radon measures, we conclude that \eqref{eq:anzellotti_young} holds.

\end{proof} 

\begin{lem} \label{lem:strict} 
Suppose that $\varphi \in \Phi_c\cap C(\Omega \times [0, \infty))$ satisfies \ref{A0}, \ref{RVA1} and \ref{aDec}, and $\varphi^*$ satisfies \ref{VA1}. Let $v \in D_\varphi$ and $\xi \in  X_{\varphi^*,2}$. If $v^k \to v$ in $L^2(\Omega)$ and $\cE_{\varphi}(v^k) \to \cE_{\varphi}(v)$, then 
\[ (\xi, Dv^k)(\Omega) \to (\xi, Dv)(\Omega).\]
\end{lem} 

\begin{proof} 
For any $\epsilon >0$, we can find an open set $A\Subset \Omega$ such that 
\begin{equation*}
\varphi(|Dv|)(\Omega \backslash A)<\frac{\epsilon}{2}  \quad \text{and} \quad \int_{\Omega \backslash A} \varphi^{*}(x,|\xi|) \, \dd x < \frac{\epsilon}{2}. 
\end{equation*}
Let $v^{k}\in C^{\infty}(\Omega)$ be the sequence from \cite[Proposition 7.1]{EHH} converging to $v$ in the sense that $v^{k}\to v$ in $L^{2}$ and $\cE_{\varphi} (v^{k})\to \cE_{\varphi}(v)$. If we consider $g\in C^{\infty}_{c}(\Omega)$ such that $0\le g \le 1$ in $\Omega$ and $g \equiv 1$ in $A$, then we have
 \begin{equation}\label{estdif}
\bigg| \int_{\Omega} (\xi,Dv^{k})-\int_{\Omega} (\xi,Dv) \bigg| \le \bigg| \langle (\xi,Dv^{k}),g \rangle - \langle (\xi,Dv),g \rangle \bigg| 
+   \int_{\Omega} (1-g)|(\xi,Dv^{k})|+\int_{\Omega} (1-g)|(\xi,Dv)|
 \end{equation}
Using \eqref{eq:anzellotti_young} we can estimate
\begin{equation*}
\int_{\Omega} (1-g)|(\xi,Dv)| \le \int_{\Omega \backslash A} |(\xi,Dv)| \le \varphi(|Dv|)(\Omega \backslash A) + \int_{\Omega \backslash A} \varphi^{*}(x,|\xi|) \, \dd x < \frac{\epsilon}{2} +\frac{\epsilon}{2} = \epsilon  
\end{equation*}
and 
\begin{align*}
\limsup_{k \rightarrow \infty} \int_{\Omega} (1-g)|(\xi,Dv^{k})| &\le \limsup_{k}\int_{\Omega \backslash A} |(\xi,Dv^{k})| \\ 
&\le \varphi(|Dv|)(\Omega \backslash A) + \int_{\Omega \backslash A} \varphi^{*}(x,|\xi|) \, \dd x < \frac{\epsilon}{2} +\frac{\epsilon}{2} = \epsilon .
\end{align*}
Since $\epsilon >0$ was arbitrary and $\langle (\xi,Dv^{k}),g \rangle \to \langle (\xi,Dv),g \rangle$, we have that
\begin{equation*}
  \bigg| \int_{\Omega} (\xi,Dv^{k})-\int_{\Omega} (\xi,Dv) \bigg| \to 0  
\end{equation*}
as $k\to \infty $ and the proof is completed.
\end{proof} 


\begin{thm}[Gauss--Green formula]\label{thm:gaussgreenformula}
Suppose that $\varphi \in \Phi_c\cap C(\Omega \times [0, \infty))$ satisfies \ref{A0}, \ref{RVA1} and \ref{aDec}, and $\varphi^*$ satisfies \ref{VA1}. Let $v \in D_\varphi$ and $\xi \in X_{\varphi^*, 2}$. Then,
\begin{equation*}
\int_\Omega (\xi, Dv) + \int_\Omega v \Div \xi \, \dd x = 0
\end{equation*}
provided that $\xi \cdot \nu = 0$.
\end{thm}




\begin{proof}
Take a sequence $v^{k}\in C^{\infty}(\overline{\Omega})$ converging to $v$ as in the statement of Lemma \ref{lem:approx_improv}. By the Gauss--Green formula \eqref{Green} in the definition of normal trace, for $\xi \in X_{\varphi^{*},2}$ with $\xi \cdot \nu = 0$ we have 
\begin{equation*}
\int_{\Omega} \xi \cdot \nabla v^{k} \, \dd x + \int_{\Omega}  v^{k} \Div \xi \, \dd x=0.
\end{equation*}
By Lemma \ref{lem:strict}, we get
\begin{equation*}
\lim_{k\to \infty} \int_{\Omega} \xi \cdot \nabla v^{k} \, \dd x = \lim_{k\to \infty} \int_{\Omega} (\xi,Dv^{k}) = \int_{\Omega} (\xi,Dv) .
\end{equation*}
Since $v^{k}\to v$ in $L^{2}(\Omega)$ and $\Div \xi \in L^{2}(\Omega)$, we have 
\begin{equation*}
\lim_{k\to \infty } \int_{\Omega} v^{k} \Div \xi    \, \dd x = \int_{\Omega} v  \Div \xi    \, \dd x 
\end{equation*}
and the proof is completed.
\end{proof}

As a consequence, we get the second part of the anticipated characterisation of the subdifferential of $\cE_\varphi$.

\begin{thm}\label{thm:localcharacterisation}
Assume that $\varphi \in \Phi_c\cap C(\Omega \times [0, \infty))$ satisfies \ref{A0}, \ref{aDec} and \ref{RVA1}, and that $\varphi^*$ satisfies \ref{VA1}. Then, the following conditions are equivalent:
\begin{itemize}
\item[(a)] $w \in \partial \cE_\varphi(v)$;

\item[(b)] $v \in BV(\Omega)$ and there exists $\xi \in X_{\varphi^*,2} \subset L^2(\Omega)^n$ with $\mathrm{div} \, \xi = w$ and $\xi \cdot \nu = 0$ such that
\begin{equation}\label{eq:equalityincharacterisation2}
\int_\Omega v w \, \dd x = \int_\Omega \varphi(|Dv|) + \int_\Omega \varphi^*(x,|\xi|) \, \dd x;
\end{equation}

\item[(c)] $v \in BV(\Omega)$ and there exists $\overline{\xi} \in  X_{\varphi^*,2} \subset L^2(\Omega)^n$ with $-\mathrm{div} \, \overline{\xi} = w$ and $\overline{\xi} \cdot \nu = 0$ such that
\begin{equation}\label{eq:equalityincharacterisation3}
\int_\Omega (\overline{\xi}, Dv) = \int_\Omega \varphi(|Dv|) + \int_\Omega \varphi^*(x,|\overline{\xi}|) \, \dd x;
\end{equation}

\item[(d)] $v \in BV(\Omega)$ and there exists $\overline{\xi} \in X_{\varphi^*,2} \subset L^2(\Omega)^n$ with $-\mathrm{div} \, \overline{\xi} = w$ and $\overline{\xi} \cdot \nu = 0$ such that
\begin{equation}\label{eq:equalityincharacterisation4}
(\overline{\xi}, Dv)  = \varphi(|Dv|) + \varphi^*(\cdot,|\overline{\xi}|) \, \mathcal{L}^n \qquad \mbox{as measures}.
\end{equation}
\end{itemize}
\end{thm}

\begin{proof}
The equivalence $(a) \Leftrightarrow (b)$ was shown in Theorem \ref{thm:abstractcharacterisation}; it remains to show the equivalences $(b) \Leftrightarrow (c) \Leftrightarrow (d)$. We first show the implication $(b) \Rightarrow (c)$. Observe that in order for equality \eqref{eq:equalityincharacterisation2} to hold, we necessarily have $v \in D_\varphi$, and take $\xi$ as in the statement; by Theorem~\ref{thm:gaussgreenformula}, the left-hand side in equation \eqref{eq:equalityincharacterisation2} can be written as
\begin{equation*}
\int_\Omega v w \, \dd x = \int_\Omega v \, \mathrm{div} \, \xi \, \dd x = - \int_\Omega (\xi, Dv) = \int_\Omega (-\xi, Dv),
\end{equation*}
and entering this expression into \eqref{eq:equalityincharacterisation2} we get that
\begin{equation*}
\int_\Omega (-\xi, Dv) = \int_\Omega \varphi(|Dv|) + \int_\Omega \varphi^*(x,|\xi|) \, \dd x,
\end{equation*}
which shows \eqref{eq:equalityincharacterisation3} once we choose $\overline{\xi} = - \xi$.

In the other direction, to show the implication $(c) \Rightarrow (b)$, take $\overline{\xi}$ as in the statement. Again, in order for equality \eqref{eq:equalityincharacterisation3} to hold, we necessarily have $v \in D_\varphi$, so applying the Gauss--Green formula again yields
\begin{equation*}
\int_\Omega (\overline{\xi}, Dv) = - \int_\Omega v \, \mathrm{div} \, \overline{\xi} \, \dd x = \int_\Omega v \, ( - \mathrm{div} \, \overline{\xi}) \, \dd x = \int_\Omega v w \, \dd x,
\end{equation*}
and entering this expression into \eqref{eq:equalityincharacterisation3} gives the desired equality \eqref{eq:equalityincharacterisation2}.

The implication $(d) \Rightarrow (c)$ is immediate and follows from integrating the condition \eqref{eq:equalityincharacterisation4} over $\Omega$. In the other direction, recall that the Young inequality \eqref{eq:anzellotti_young} states that for every Borel set $B$ we have the inequality
\begin{equation} \label{eq:youngB}
\int_B (\overline{\xi}, Dv)  \leq \int_B |(\overline{\xi}, Dv)|  \leq  \int_B \varphi(|D v|) + \int_B \varphi^*(x,|\overline{\xi}|) \, \dd x.
\end{equation}
In particular, the same inequality holds with $\Omega \setminus B$ in the place of $B$. If there were $B$ such that inequality \eqref{eq:youngB} is strict, adding the inequalities for $B$ and $\Omega \setminus B$ would lead to
\begin{equation*}
\int_\Omega (\overline{\xi}, Dv)  <  \int_\Omega \varphi(|D v|) + \int_\Omega \varphi^*(x,|\overline{\xi}|) \, \dd x = \int_\Omega (\overline{\xi}, Dv), 
\end{equation*}
where the last equality follows from \eqref{eq:equalityincharacterisation3}. The obtained contradiction shows that \eqref{eq:youngB} is an equality for all $B$ Borel, i.e., \eqref{eq:equalityincharacterisation4} holds.
\end{proof}

\begin{remark}
For the special choice $\varphi \in \Phi_{c}$ with $\varphi(t) = t$, we have the classical total variation 
$$\cE(v)=\int_{\Omega}|Dv|.$$
It is clear that such $\varphi$ satisfies the conditions \ref{A0} for any $\beta \in (0,1]$, \ref{aDec} for any $q\ge 1$ and \ref{RVA1} for any modulus of continuity. Moreover, the conjugate 
$$\varphi^*(t)= \infty \chi_{(1,\infty)}(t)$$
satisfies \ref{VA1} for any modulus of continuity. Hence, applying Theorem \ref{thm:localcharacterisation}, we recover the characterisation of \cite[Proposition 1.10]{ACMBook} for $v\in BV(\Omega)\cap L^2(\Omega)$ and a vector field in the class 
\[X_{\infty,2}=\{\xi\in L^{\infty}(\Omega)^n : \operatorname{div}\xi \in L^2(\Omega)\}.\]
In point (b), the condition $\int_{\Omega} \varphi^{*}(x,|\xi|) \, \dd x < +\infty$ and the 1-homogeneity of $\varphi$ implies that $\varphi^*(|\xi|)=0$~a.e., which translates to $\|\xi\|_{\infty}\le 1$ (and similarly for $\overline{\xi}$ in points (c) and (d)). 

As a result, one may reformulate the equivalence between points (a) and (d) of Theorem \ref{thm:localcharacterisation} for the isotropic total variation as the following statement: $w \in \partial \mathcal{E}(v)$ if and only if there exists $\overline{\xi} \in  L^\infty(\Omega)^n$ with $\| \overline{\xi} \|_\infty \leq 1$ and $\mathrm{div} \, \overline{\xi} \in L^2(\Omega)^n$ such that the following conditions are satisfied:
\begin{equation*}
-\mathrm{div} \, \overline{\xi} = w;
\end{equation*}
\begin{equation*}
\overline{\xi} \cdot \nu = 0;
\end{equation*}
and
\begin{equation*}
(\overline{\xi}, Dv) = |Dv|,
\end{equation*}
which recovers the result of \cite[Proposition 1.10]{ACMBook}. From this, one recovers the optimality conditions for the ROF problem via relation \eqref{EL_abstract}.

\end{remark}

\begin{remark}
The characterisation of $\cE_\varphi$ given in Theorem \ref{thm:localcharacterisation} also allows us to characterise the gradient flow of this functional in an explicit way. By the Brezis--K\={o}mura theorem, see \cite{Brezis0,Komura} for the original results and \cite{Brezis} for an overview, the abstract Cauchy problem
\begin{equation}\label{eq:gradientflow}
\left\{ \begin{array}{ll} 0 \in \frac{\dd u}{\dd t} + \partial \cE_\varphi (u(t)) \, \quad & \text{if } t \in (0, T);
  \\[10pt] u(0) = u_0 \in L^2(\Omega) \quad & \end{array} \right.
\end{equation}
has a unique strong solution $u \in C([0,T]; L^2(\Omega)) \cap W^{1,2}_{\rm loc}((0, T]; L^2(\Omega))$. By the characterisation of $\cE_\varphi$, $u$ is a solution to \eqref{eq:gradientflow} if and only if for almost all $t \in (0,T)$ we have $u(t) \in BV(\Omega)$ and there exist vector fields $\overline{\xi}(t) \in  X_{\varphi^*,2}$ such that the following conditions hold:
$$ u_t(t,\cdot) = \mathrm{div}(\overline{\xi}(t)) \quad \text{in } \ \Omega, $$
$$ (\overline{\xi}(t), Du(t)) = \varphi(|Du(t)|) + \varphi^*(\cdot,|\overline{\xi}|(t)) \, \mathcal{L}^n \quad \text{as measures},$$
$$ \overline{\xi}(t) \cdot \nu = 0. $$
Again, for the choice $\varphi(t) = t$, we recover the characterisation of solutions to the Neumann problem for the total variation flow given in \cite{ACMBook}.
\end{remark}

\section{Particular choices of $\varphi$}\label{sec:particularcases}

In this Section, we discuss some of the assumptions used in this paper in the context of the double-phase functional and the variable exponent case.

\begin{remark}[Conjugates]\label{rem:conjugates}
For reference, let us state the form of convex conjugates of certain $\Phi$-functions relevant to this paper, which may be computed using \cite[Lemma 3.5]{EHH}. For the double-phase case, i.e.,
$$\varphi_{\rm dp}(x,t) = t + a(x) \frac{t^{q}}{q},$$
we obtain the expression
$$\varphi_{\rm dp}^{*}(x,s) = \twopartdef{\infty \chi_{(1,\infty)}(s)}{\mbox{if } a(x) = 0;}{a(x)^{1-q'}\frac{(s-1)_{+}^{q'}}{q'}}{\mbox{if } a(x) > 0.}$$ 
In the variable exponent case, i.e.,
$$\varphi_{\rm var}(x,t)=\frac{1}{p(x)}t^{p(x)},$$ 
it holds that
$$\varphi_{\rm var}^{*}(x,s)= \twopartdef{\infty \chi_{(1,\infty)}(s)}{\mbox{if } p(x) = 1;}{\frac{1}{p'(x)}s^{p'(x)}}{\mbox{if } p(x) > 1.}$$


\end{remark}

Clearly, the condition \ref{A0} is satisfied, and conditions which ensure that they satisfy \ref{aDec} were identified e.g.\ in \cite{HarjulehtoHasto}. We therefore concentrate the discussion below on the conditions \ref{RVA1}, which is required for the $\Phi$-function $\varphi$, and the condition \ref{VA1} which is needed for its dual.

First we discuss condition \ref{RVA1}. In the literature, the variable exponent case is usually mentioned without the coefficient $\frac{1}{p(x)}$. In \cite[Proposition 3.2]{EHH} it is shown that the $\Phi$-function
$$\varphi_{\rm var}(x,t)=t^{p(x)}$$
satisfies the condition \ref{RVA1} if and only if $\frac{1}{p}$ is log-H\"{o}lder continuous, i.e.,
\begin{equation} \label{eq:lhc}
\bigg| \frac{1}{p(x)} - \frac{1}{p(y)} \bigg| \leq \frac{C}{\log(e + \frac{1}{|x-y|})}
\end{equation}
with $C>0$, and the following strong H\"{o}lder continuity condition for $\frac{1}{p}$:
\begin{equation} \label{eq:slhc1}
\lim_{x\to y} \bigg| 1 - \frac{1}{p(x)} \bigg| \log \frac{1}{|x-y|} = 0
\end{equation}
holds uniformly in $y \in \{ p = 1 \}$. In the next Proposition, we show a similar characterisation for the double phase functional.

\begin{prop}[Condition \ref{RVA1} for double-phase] 
The $\Phi$-function 
$$\varphi_{\mathrm{dp}}(x,t)=t+a(x)t^{q}$$
satisfies the condition \ref{RVA1} if and only if the weight function $a$ satisfies the almost H\"{o}lder continuity condition
\begin{equation}\label{eq:almostholder}
a(y) \leq C (a(x) + |x - y|^{n(q-1)})  \end{equation}
and is strongly H\"{o}lder continuous in $\{ a = 0 \}$, i.e., 
\begin{equation}\label{StrongHolder}
\lim_{x\to y} \frac{|a(x)|}{|x-y|^{n(q-1)}}=0
\end{equation}
uniformly in $y\in \{a=0\}.$    
\end{prop}

\begin{proof} The relation of the condition \ref{A1} and almost H\"{older} continuity of the weight function \eqref{eq:almostholder} has been proved in \cite[Proposition 7.2.2]{HarjulehtoHasto} (see also \cite{BCFM} for a more detailed discussion of this condition). Thus, we will focus on the strong H\"{o}lder continuity on the set $\{a=0\}$.
Suppose that $\varphi_{\mathrm{dp}}$ satisfies \ref{RVA1}; since the inequality in the condition holds for any $K > 0$, let us fix $K = 1$ and the corresponding modulus of continuity $\omega$. Let $y\in \{a=0\}$ or equivalently $\varphi^{'}_{\infty}(y)<\infty$, then $\varphi_{\mathrm{dp}}=t$. If we choose $t=|x-y|^{-n}\ge 1$ and set $r:=|x-y|$, we get 
$$\frac{t}{1+\omega(r)}+a(x)\left(\frac{t}{1+\omega(r)}\right)^{q}=\varphi_{\mathrm{dp}} \bigg(x,\frac{t}{1+\omega(r)} \bigg) \le \varphi_{\mathrm{dp}}(y,t)+\omega(r)\le (1+\omega(r))t,
$$
from which we have that $1+a(x)t^{q-1}\le (1+\omega(r))^{q+1}$ and this implies
$$
1\le 1+\frac{a(x)}{|x-y|^{n(q-1)}}\le (1+\omega(r))^{q+1}\to 1
$$
as $r\to 0^{+}$. Thus, $$\lim_{x\to y} \frac{|a(x)|}{|x-y|^{n(q-1)}}=0.$$
Conversely, assume that the weight $a$ satisfies \eqref{StrongHolder}. We will show that $\varphi_{\mathrm{dp}}$ satisfies \ref{RVA1}. We have that
$$
\omega_{a}(r):= \sup_{y\in \{a=0\}, x\in B_{r}(y)}  \frac{|a(x)|}{|x-y|^{n(q-1)}}\to 0
$$
as $r\to 0^{+}$. To prove the \ref{RVA1} when $a(y)=0$, we will show that
$$
\varphi \bigg(x,\frac{t}{1+\omega(r)} \bigg) \leq \varphi(y,t),
$$
when $\varphi(y,t)\in [0,\frac{K}{|x-y|^{n}}]$, for some modulus of continuity $\omega(r)$ (depending on $K>0$). The above inequality follows if we show that $
1+a(x)t^{q-1} \le 1+\omega(r).
$ The inequality holds trivially if $t\in [0,1]$. For $t>1$, since the function $t\mapsto 1+a(x)t^{q-1}$ is increasing, it is enough to check for $t=\frac{K}{|x-y|}$.
We have
$$
1+a(x)t^{q-1}\le 1+ a(x) \frac{K^{q-1}}{|x-y|^{n(q-1)}} \le 1+ K^{q-1}\omega_{a}(r), 
$$
thus choosing $\omega(r):=  K^{q-1} \omega_{a}(r)$ the claim holds.
Now if $a(x)=0$ in the \ref{RVA1} condition then
$$
\frac{t}{1+\omega(r)}\le t \le t+a(y)t^{q}+\omega(r)
$$
holds for some modulus of continuity $\omega(r)$ since $a\ge 0$.
\end{proof}

Next, we discuss the full condition \ref{VA1}. In fact, there are several conditions of this type appearing in the literature; the first one was given in \cite{HastoOk}. More precisely, in \cite[Definition 4.1]{HastoOk} the following version of \ref{VA1} which includes a comparison of values of $\varphi$ on small balls was introduced: 

\begin{enumerate}
\item[(VA1-ball)]\labeltext{(VA1-ball)}{VA1-ball}
   There exists a non-decreasing continuous function $\omega: \left[ 0,\infty \right)\to \left[0,1 \right]$ with $\omega(0)=0$ such that for any small ball $B_{r} \Subset \Omega$
 \begin{equation}\label{condition}
      \varphi_{B_{r}}^{+}(t) \le (1+\omega(r)) \varphi_{B_{r}}^{-}(t), \text{for all } t>0 \text{ with } \varphi_{B_{r}}^{-}(t) \in \left[\omega(r),|B_{r}|^{-1}  \right],
   \end{equation}
where $ \varphi_{B_{r}}^{+}(t)=\sup_{x\in B_{r}}\varphi(x,t)$   and $ \varphi_{B_{r}}^{-}(t)=\inf_{x\in B_{r}}\varphi(x,t)$.
\end{enumerate}

In \cite[section 8]{HastoOk} the authors provide examples of this continuity condition in several specific settings, including the perturbed homogeneous case, the variable exponent case and the double phase case. Another variant of this condition was introduced by the same authors in \cite{HastoOk22ARMA}. These definitions are not well-adapted to the linear growth case, as is made evident e.g.\ from the fact that in condition \ref{VA1-ball} the function $\omega(r)$ yields a uniform continuity estimate over a sufficiently small ball, while in the linear-growth case when the conjugate function may take the value $\infty$ we do not expect such behaviour; therefore, in this paper, we use the condition from \cite{EHH} (and subsequent approximation results which rely on this condition). In this case, the only known result on the characterisation of \ref{VA1} concerns the variable exponent case. Again without the coefficient $\frac{1}{p(x)}$, it is claimed in \cite{HJR} that for the $\Phi$-function $\varphi_{\rm var}(x,t)=t^{p(x)}$ (a reformulation of) \ref{VA1} is implied by the strong $\log$-H\"older condition, i.e.,
\begin{equation}\label{eq:slhc}
\bigg| \frac{1}{p(x)} - \frac{1}{p(y)} \bigg| \leq \frac{\omega(|x-y|)}{\log(e + \frac{1}{|x-y|})}
\end{equation}
for some modulus of continuity $\omega$. 

However, let us note that condition \ref{VA1} (for the conjugate) is only needed in this paper in order to apply the approximation results from \cite{EHH}. In \cite[Lemma 3.3]{HHLT}, the existence of $W^{1,\varphi}$-class approximation was obtained (in the case of a rectangular domain) under the assumptions that $\frac{1}{p}$ is log-H\"{o}lder continuous (i.e.\ \eqref{eq:lhc} holds) and \eqref{eq:slhc1} holds for $y\in  \{p=1\}.$ This sequence can be further smoothed by \cite[Corollary 4.6]{Juusti}. Thus, in the variable exponent case in a rectangular domain, our results---including Theorem \ref{thm:localcharacterisation}---hold assuming that \eqref{eq:lhc} is satisfied and \eqref{eq:slhc1} holds uniformly in $y \in \{p=1\}$. For the double-phase functional a different set of assumptions on the weight $a$ which ensures existence of approximations was given in \cite{HarjulehtoHasto2021}. \\


\noindent {\bf Acknowledgments.} The work of the first author has been supported by the Austrian Science Fund (FWF), grants 10.55776/ESP88 and 10.55776/I5149. The second and third authors received support from the grant 2020/36/C/ST1/00492 of the National Science Centre (NCN), Poland. The third author has been supported by the Special Account for Research Funding of the National Technical University of Athens. The third author is also grateful for the hospitality and support of the Institute of Mathematics of the Polish Academy of Sciences (IM PAN) in Warsaw, where part of this work was carried out. For the purpose of open access, the authors have applied a CC BY public copyright licence to any Author Accepted Manuscript version arising from this submission.

\bibliographystyle{asdfgh}
\bibliography{bib.bib}

\begin{thebibliography}{10}
\providecommand{\url}[1]{\texttt{#1}}
\providecommand{\urlprefix}{URL }
\expandafter\ifx\csname urlstyle\endcsname\relax
  \providecommand{\doi}[1]{doi:\discretionary{}{}{}#1}\else
  \providecommand{\doi}{doi:\discretionary{}{}{}\begingroup \urlstyle{rm}\Url}\fi
\providecommand{\eprint}[2][]{\url{#2}}

\bibitem{AFP}
Luigi Ambrosio, Nicola Fusco, Diego Pallara.
\newblock \emph{Functions of bounded variation and free discontinuity problems}.
\newblock Oxford Mathematical Monographs. The Clarendon Press, Oxford University Press, New York (2000).

\bibitem{ABCM2001}
F.~Andreu, C.~Ballester, V.~Caselles, J.~M. Maz\'{o}n.
\newblock The {D}irichlet problem for the total variation flow.
\newblock \emph{J. Funct. Anal.} 180 (2001) pp. 347--403.
\newblock \doi{10.1006/jfan.2000.3698}.

\bibitem{ACMBook}
F.~Andreu, V.~Caselles, J.M. Mazón.
\newblock \emph{Parabolic Quasilinear Equations Minimizing Linear Growth Functionals}, vol. 223 of \emph{Progress in Mathematics}.
\newblock Birkh\"auser (2004).

\bibitem{ACM2002}
Fuensanta Andreu, Vincent Caselles, Jos\'{e}~Mar\'{\i}a Maz\'{o}n.
\newblock A parabolic quasilinear problem for linear growth functionals.
\newblock \emph{Rev. Mat. Iberoamericana} 18 (2002) pp. 135--185.
\newblock \doi{10.4171/RMI/314}.

\bibitem{Anzellotti}
G.~Anzellotti.
\newblock Pairings between measures and bounded functions and compensated compactness.
\newblock \emph{Ann. di Matematica Pura ed Appl. IV} 135 (1983) pp. 293--318.

\bibitem{BCFM}
Micha{\l} Borowski, Iwona Chlebicka, Filomena De~Filippis, B{\l}a\.{z}ej Miasojedow.
\newblock Absence and presence of {L}avrentiev's phenomenon for double phase functionals upon every choice of exponents.
\newblock \emph{Calc. Var. Partial Differential Equations} 63 (2024).
\newblock \doi{10.1007/s00526-023-02640-1}.

\bibitem{Brezis0}
H.~Brezis.
\newblock Monotonicity methods in {H}ilbert spaces and some applications to nonlinear partial differential equations.
\newblock In \emph{Contributions to Nonlinear Functional Analysis, Proc. Sympos. Univ. Wisconsin}. Madison, Academic Press, New York, USA (1971) pp. 101--156.

\bibitem{Brezis}
H.~Brezis.
\newblock \emph{Operateurs Maximaux Monotones}.
\newblock North Holland, Amsterdam (1973).

\bibitem{ChambolleLions}
Antonin Chambolle, Pierre-Louis Lions.
\newblock Image recovery via total variation minimization and related problems.
\newblock \emph{Numer. Math.} 76 (1997) pp. 167--188.
\newblock \doi{10.1007/s002110050258}.

\bibitem{ChambollePock}
Antonin Chambolle, Thomas Pock.
\newblock A first-order primal-dual algorithm for convex problems with applications to imaging.
\newblock \emph{J. Math. Imaging Vis.} 40 (2011) pp. 120--145.
\newblock \doi{10.1007/s10851-010-0251-1}.

\bibitem{CLR}
Y.~Chen, S.~Levine, M.~Rao.
\newblock Variable exponent, linear growth functionals in image restoration.
\newblock \emph{Siam J. Appl. Math.} 66 (2006) pp. 1383--1406.
\newblock \doi{10.1137/050624522}.

\bibitem{Dacorogna}
Bernard Dacorogna.
\newblock \emph{Direct methods in the calculus of variations}, vol.~78 of \emph{Applied Mathematical Sciences}.
\newblock Springer, New York, second edn. (2008).

\bibitem{EkelandTemam}
I.~Ekeland, R.~Temam.
\newblock \emph{Convex analysis and variational problems}.
\newblock North-Holland Publ. Company, Amsterdam (1976).

\bibitem{EHH}
Michela Eleuteri, Petteri Harjulehto, Peter H{\"a}st{\"o}.
\newblock Bounded variation spaces with generalized {O}rlicz growth related to image denoising.
\newblock \emph{Mathematische Zeitschrift} 310 (2025).
\newblock \doi{10.1007/s00209-025-03731-9}.

\bibitem{GornyMazon}
Wojciech Górny, José~M. Mazón.
\newblock A duality-based approach to gradient flows of linear growth functionals (to appear in Publicacions Matemàtiques).
\newblock \eprint{arXiv:2212.08725}.

\bibitem{HHLT}
Peter Harjulehto, Peter H{\"a}st{\"o}, Visa Latvala, Olli Toivanen.
\newblock Critical variable exponent functionals in image restoration.
\newblock \emph{Applied Mathematics Letters} 26 (2013) pp. 56--60.

\bibitem{HarjulehtoHasto}
Petteri Harjulehto, Peter H\"{a}st\"{o}.
\newblock \emph{Orlicz spaces and generalized {O}rlicz spaces}, vol. 2236 of \emph{Lecture Notes in Mathematics}.
\newblock Springer, Cham (2019).
\newblock \doi{10.1007/978-3-030-15100-3}.

\bibitem{HarjulehtoHasto2021}
Petteri Harjulehto, Peter H\"{a}st\"{o}.
\newblock Double phase image restoration.
\newblock \emph{J. Math. Anal. Appl.} 501 (2021).
\newblock \doi{10.1016/j.jmaa.2019.123832}.

\bibitem{HJR}
Peter H\"ast\"o, Jonne Juusti, Humberto Rafeiro.
\newblock Riesz spaces with generalized {O}rlicz growth (preprint 2022).
\newblock \eprint{arXiv:2204.14128}.

\bibitem{HastoOk}
Peter H\"ast\"o, Jihoon Ok.
\newblock Maximal regularity for local minimizers of non-autonomous functionals.
\newblock \emph{J. Eur. Math. Soc. (JEMS)} 24 (2022) pp. 1285--1334.
\newblock \doi{10.4171/JEMS/1118}.

\bibitem{HastoOk22ARMA}
Peter H\"ast\"o, Jihoon Ok.
\newblock Regularity theory for non-autonomous partial differential equations without {U}hlenbeck structure.
\newblock \emph{Arch. Rational Mech. Anal.} 245 (2022) pp. 1401--1436.
\newblock \doi{10.1007/s00205-022-01807-y}.

\bibitem{Juusti}
Jonne Juusti.
\newblock Extension in generalized {O}rlicz-{S}obolev spaces.
\newblock \emph{J. Math. Anal. Appl.} 522 (2023).
\newblock \doi{10.1016/j.jmaa.2022.126941}.

\bibitem{Komura}
Y.~K\={o}mura.
\newblock Nonlinear semi-groups in {H}ilbert space.
\newblock \emph{J. Math. Soc. Japan} 19 (1967) pp. 493--507.

\bibitem{Moll2005}
J.~S. Moll.
\newblock The anisotropic total variation flow.
\newblock \emph{Math. Ann.} 332 (2005) pp. 177--218.
\newblock \doi{10.1007/s00208-004-0624-0}.

\bibitem{NovotnyStravskraba}
A.~Novotn\'{y}, I.~Stra\v{s}kraba.
\newblock \emph{Introduction to the mathematical theory of compressible flow}, vol.~27 of \emph{Oxford Lecture Series in Mathematics and its Applications}.
\newblock Oxford University Press, Oxford (2004).

\bibitem{ROF}
Leonid~I. Rudin, Stanley Osher, Emad Fatemi.
\newblock Nonlinear total variation based noise removal algorithms.
\newblock \emph{Phys. D} 60 (1992) pp. 259--268.
\newblock \doi{10.1016/0167-2789(92)90242-F}.
\newblock Experimental mathematics: computational issues in nonlinear science (Los Alamos, NM, 1991).

\end{thebibliography}

\end{document}